\documentclass[pdflatex,referee,sn-mathphys,Numbered]{sn-jnl}

\usepackage{graphicx}%
\usepackage{multirow}%
\usepackage{amsmath,amssymb,amsfonts}%
\usepackage{amsthm}%
\usepackage{mathrsfs}%
\usepackage[title]{appendix}%
\usepackage{xcolor}%
\usepackage{textcomp}%
\usepackage{manyfoot}%
\usepackage{booktabs}%
\usepackage{algorithm}%
\usepackage{algorithmicx}%
\usepackage{algpseudocode}%
\usepackage{listings}%
\usepackage{paralist}%

\newcommand{\re}{\mathit{re}}
\newcommand{\U}{\mathit{ue}}

\newtheorem{corollary}{Corollary}
\newtheorem{proposition}{Proposition}
\theoremstyle{remark}
\theoremstyle{definition}
\newtheorem{example}{Example}

\begin{document}

\title[Scheduling and dimensioning of energy stores]{Scheduling and
  dimensioning of heterogeneous energy stores, with applications to
  future GB storage needs}


\author*{\fnm{Stan} \sur{Zachary}}\email{s.zachary@gmail.com}



\affil{\orgdiv{School of Mathematical \& Computer Sciences},
  \orgname{Heriot-Watt
    University},
  \orgaddress{\city{Edinburgh}, \postcode{EH14 4AS}, \country{UK}}}



\abstract{ Future ``net-zero'' electricity systems in which all or
  most generation is renewable may require very high volumes of
  storage in order to manage the associated variability in the
  generation-demand balance.  The physical and economic
  characteristics of storage technologies are such that a mixture of
  technologies is likely to be required.  This poses nontrivial
  problems in storage dimensioning and in real-time management.  We
  develop the mathematics of optimal scheduling for system adequacy,
  and show that, to a good approximation, the problem to be solved at
  each successive point in time reduces to a linear programme with a
  particularly simple solution.  We argue that approximately optimal
  scheduling may be achieved without the need for a running forecast
  of the future generation-demand balance.  We consider an extended
  application to GB storage needs, where savings of tens of billions
  of pounds may be achieved, relative to the use of a single
  technology, and explain why similar savings may be expected
  elsewhere.  }



\maketitle

\newpage

\textbf{Notation}

\begin{tabular}[ht]{ll}
  $t$      & time (discrete)\\
  $S$      & set of stores\\
  $E_i$    & capacity of store~$i\in S$\\ 
  $P_i$    & maximum input power of store~$i\in S$\\ 
  $Q_i$    & maximum output power of store~$i\in S$\\ 
  $\eta_i$ & round-trip efficiency of store~$i\in S$\\ 
  $s_i(t)$ & level of energy in store~$i\in S$ at time~$t$\\ 
  $s(t)$   & vector of levels $(s_i(t),\, i\in S)$ (i.e.\ state of system at
  time~$t$)\\ 
  $r_i(t)$ & rate at which energy is added to store~$i\in S$ at time~$t$\\ 
  $r(t)$   & vector of rates $(r_i(t),\, i\in S)$\\ 
  $\re(t)$ & residual energy (surplus of generation over demand) at time~$t$\\ 
  $u(t)$   & imbalance at time~$t$ (see equation~\eqref{eq:4})\\ 
  $\U(t)$  & total unserved energy up to time~$t$\\ 
  $V^t(s)$ & value function defined on states $s$ at time~$t$\\ 
  $v^t_i(s)$ & partial derivative of $V^t(s)$ with respect to
               $i^{th}$ component of $s$\\ 
  $\lambda_i$ & scale parameter of $v^t_i(s)$ (see equation~\eqref{eq:10})
\end{tabular}

\bigskip

\textbf{Abbreviations}

\begin{tabular}[ht]{ll}
  ACAES & advanced (adiabatic) compressed air energy storage\\
  GGDDF & greedy greatest-discharge-duration-first policy\\
  GRTEF & greatest-round-trip-efficiency first policy
\end{tabular}

\newpage

\section{Introduction}
\label{sec:introduction}


Future electricity systems in which all or most generation is
renewable, and hence highly variable, may require extremely high
volumes of storage in order to manage this variability and to ensure
that demand may always be met.  Detailed assessments of such needs,
under a future ``net-zero'' carbon emissions strategy and on the
assumption that generation overcapacity is not uneconomically large,
are given for GB
by~\cite{Cardenas2021,RoulstoneCosgrove2022,RoyalSociety2023,RoyalSociety2023_supp},
for Germany by~\cite{Ruhnau2022}, and for the US by~\cite{Shaner2018}.
In each of these cases storage needs to be sufficient to be able to
meet several or many weeks of demand, requiring many tens of
terawatt-hours of storage with capital costs which may run into many
tens, or even hundreds, of billions of (US) dollars.  Similar
conclusions for many other countries may be deduced from the results
of~\cite{Tong2021,CHENM2018,BlancoFaaij2018}.  Further discussion and
references are given by~\cite{CRZ2023}.

In environments in which most generation is renewable, and hence
highly variable, energy may have to be stored over long periods of
time.  In northern European countries, for example, the output of wind
generation varies significantly from year to year
(see~\cite{StaffellPfenninger2016}), necessitating storage of excess
energy in abundant years for use in what would otherwise be lean years
and leading to the high storage capacity requirements referenced
above.  For the reasons we explain below, a mixture of storage
technologies is likely to be required.  The problems considered in the
present paper are those of the scheduling and dimensioning of such
storage, with the objective of meeting energy demand to a required
reliability standard as economically as possible.  In particular, in
the real-time scheduling, or management, of such storage, the
information available for decision-making at each point in time
consists of the current state of system together with some
description, which is at best probabilistic, of the likely evolution
of the future supply and demand processes to be managed by that
storage.  What is \emph{not} available is detailed and precise
\emph{foresight} of the future supply and demand processes.  As we
show in the present paper (Example~\ref{ex:2}), the assumption of such
foresight in considering long-term energy storage may lead to a
considerable underestimation of storage requirements.  The problem of
the \emph{real-time} management of long-term energy
storage---particularly when this utilises multiple technologies---is
only rarely touched upon in the existing, and very large, storage
literature (see below) and it is this gap which the present paper
seeks to address.  The long-term scheduling problems we consider are
nontrivial in that, if storage is not managed properly, then it is
likely that there will frequently arise the situation in which there
is sufficient energy in storage to meet demand but it is located in
too few stores for it to be possible to serve it at the required rate.

The problems we consider are concerned with \emph{system adequacy} and
are thus considered from a \emph{societal} viewpoint, which generally
coincides with that of the electricity system operator.  This is in
contrast
the viewpoint of a storage provider seeking to maximise profits---for
which there exists a substantial literature (see,
e.g.~\cite{SDJW,CZ,DEKM,GTL,PADS,WEITZEL2018582} and the many
references therein).  The \emph{societal} problem of managing energy
systems, as defined above, also has a large associated literature.
However, this is mostly in the context of \emph{short-term} storage
used to cover occasional periods of generation shortfall, often with
sufficient time for recharging between such periods (see, e.g.\
\cite{NGECR,KHAN201839,Sioshansi2014,ZHOU201512}).  Alternatively, the
literature is concerned with the management of microsystems
(see~\cite{Choudhury2022} for a comprehensive review) or with
multi-objective problems~\cite{CDHB2019,CDHB2020}.  In nearly all of
this literature, it appears that foresight (as defined above) is
assumed and the optimal control strategy is determined on this basis.
With relatively short-term problems this may well be reasonable.

Long-term storage is also considered in the existing literature
(see~\cite{MIT2022,Ruhnau2022,Shaner2018,Tong2021}.  However, the
purpose of such studies is generally the determination of overall
storage requirements.  Such studies typically start with one or more
years of supply-demand data; a (mixed-integer) linear programming
approach is then used to simultaneously dimension and schedule
storage.  Frequently this happens via the use of economic capacity
expansion models (see, e.g., \cite{PLEXOS, GenX}).  What is of
interest in such studies is the dimensioning.  The scheduling cannot
be implemented in practice (except in trivial problems) since the
approaches used in such studies again generally assume foresight.

A further disadvantage of those approaches which assume foresight as
above is that the complexity of the numerical computation involved
typically grows far faster than linearly in the length of the data
series used to fit the models.  This means that such studies typically
can only consider data series of a single or very few years, whereas
considerably longer data series are required for the correct
dimensioning of long-term storage in particular---see,
e.g.,~\cite{RoyalSociety2023} for a discussion of this issue.
Further, it is often necessary to use approximation techniques which
consider a succession of timescales.

The urgent need for a solution of these scheduling problems is at
least implicit in many of the above references and is highlighted by
the recent Royal Society report~\cite{RoyalSociety2023} to the UK
government on long-term energy storage.  In the production of that
report, no satisfactory method of long-term scheduling (which did not
assume foresight) was available in the existing literature.  The
mathematics of present paper---along with the alternative approach of
the paper~\cite{CosgroveRoulstone2021}---was developed to fill this
gap and was used as the basis of the scheduling of multiple storage
technologies in the Royal Society supplementary
report~\cite{RoyalSociety2023_supp}.  The latter report also compares
the approach of~\cite{CosgroveRoulstone2021} with that of the present
paper in some detail---see, in particular, Table SI 3.3.  The present
approach results in considerably smaller storage power requirements
than those of~\cite{CosgroveRoulstone2021}, as stores effectively
share their power capabilities---at the occasional expense of higher
capacity requirements.  A further comparison is given by the
paper~\cite{CRZ2023}.


In a large system, such as that of an entire country, the processes of
demand and of renewable generation vary on multiple timescales: on a
timescale measured in hours there is diurnal variation in demand and
in solar generation; on a timescale of days and weeks there is
weather-related variation in demand and in most forms of renewable
generation, and there is further demand variation due to weekends and
holiday periods; on longer timescales there is seasonal variation in
both demand and renewable generation which may extend to major
differences between successive years (see, e.g.,
\cite{Cardenas2021,RoyalSociety2023,Shaner2018,Tong2021,Ruhnau2022}).
Variation in the generation-demand balance may be managed by a number
of different storage technologies.  These vary greatly in their costs
per unit of capacity and per unit of power (the maximum rate at which
they may be charged or discharged), and further in their
\emph{(round-trip) efficiencies} (energy output as a fraction of
energy input).  In consequence, different storage technologies are
typically appropriate to managing variation on these different
timescales---see~\cite{Cardenas2021,RoulstoneCosgrove2022} for some
detailed comparisons and analysis, and also the recent MIT
report~\cite{MIT2022} (especially Figure~1.6) for a discussion of
differing technology costs and their implications.  We further explore
these issues in Section~\ref{sec:appl-gb-energy}.

It thus seems likely that---as previously remarked---in the management
of such future electricity systems, there will be a need for a mix of
storage technologies.  This will be such that most of the required
storage \emph{capacity} will be provided by those technologies such as
chemical storage, which, despite low efficiency and high input-output
(power) costs, are able to provide this capacity most economically,
while a high proportion of the \emph{power} requirements will be met
by technologies such as batteries or advanced (adiabatic) compressed
air energy storage (ACAES) with much higher efficiencies and lower
power costs.  For example, if chemical storage as above were also used
to manage shorter-term variation, necessitating frequent and more
rapid energy input and output, its low efficiency would greatly drive
up its capacity requirement.
(Although there is considerable uncertainty in
future costs, we show that, for the GB case study of
Section~\ref{sec:appl-gb-energy} and on the basis of those costs given
by~\cite{MIT2022}, Figure~1.6, or by~\cite{Cardenas2021,
  RoulstoneCosgrove2022}, it is likely that the use of an appropriate
mixture of storage technologies would result in cost savings of the
order of many billions of pounds, compared with the use of the most
economical single technology.  
Similar results are to be expected for other countries.)

There now arise the questions of how such storage may be economically
dimensioned, and of how it may be managed in real-time, i.e.\ without
\emph{foresight}.  The ability to answer the former question depends
on having a sufficiently good understanding of the answer to the
latter.  The problem of managing, or scheduling, any given set of
stores is that of deciding by how much each individual store should be
charged or discharged at each successive point in time in order to
best manage the generation-demand (im)balance---usually with the
objective of minimising total unmet demand, or \emph{unserved energy},
over some given period of time.
In this context, usually those stores corresponding to any given
technology may be treated as a single store provided their
capacity-to-power ratios are approximately equal (see
Section~\ref{sec:model}).  However, as discussed above, the scheduling
problem is a real-time problem, and in deciding which storage
technology to prioritise at any given point in time, it is difficult
to attempt to classify the current state of the generation-demand
balance as representing short, medium, or long-term variation.  Within
the existing literature,
\cite{CosgroveRoulstone2021} uses a heuristic algorithm to attempt such
a decomposition, while~\cite{Cardenas2021} uses a filtering approach to
choose between medium- and long-term storage.  (Neither of these
approaches allows for cross-charging---see Section~\ref{sec:model}.)


In the present paper the above problem is formulated as one in
mathematical optimisation theory in order to derive policies in which
cooperation between stores happens automatically when this is
beneficial, thereby enabling given generation-demand balance processes
to be managed by storage systems which are considerably more compactly
dimensioned in their power requirements in particular (see
also~\cite{RoyalSociety2023}).  Section~\ref{sec:model} of the paper
defines the relevant mathematical model for the real-time management
of multiple stores in the absence of \emph{foresight}.  This
incorporates capacity and rate (power) constraints, together with
round-trip efficiencies, and allows for entirely general scheduling
policies.  Section~\ref{sec:nature-optim-solut} develops the relevant
mathematics for the identification of optimal policies, when the
objective is the minimisation of cumulative unserved energy to each
successive point in time.  We show that it is sufficient to consider
policies that are \emph{greedy} in an extended sense defined there.
We further show that, at each successive point in time, the scheduling
problem may be characterised as that of maximising a \emph{value
  function} defined on the space of possible energy levels of the
stores, and that
the optimisation problem to be solved at that time is approximately a
(small) linear programme, with a simple, non-iterative, solution.  We
give conditions under which it is possible to find optimal policies,
exact or approximate, from within the class of \emph{non-anticipatory}
policies, i.e.\ those which do not require real-time \emph{foresight}
of the generation-demand balance.
Section~\ref{sec:appl-gb-energy} considers an extended application to
future GB energy storage needs, which
aims to be as realistic as possible.  (Again similar results are to be
expected for many other countries.)  The aims are both to demonstrate
the applicability of the present theory, and further to show how one
might reasonably go about solving the practical problems of
identifying, dimensioning and managing future storage needs.
We demonstrate the general success, and occasional limitations, of
\emph{non-anticipatory} policies as defined above.
The concluding Section~\ref{sec:conclusions} considers some practical
implications of the preceding results.     We also indicate
briefly how the analysis might be extended to include network
constraints, if desired, although we also indicate why these are less
significant in the context of dimensioning long-term storage.

\section{Model}
\label{sec:model}

We study the management over (discrete) time of a set~$S$ of stores,
where each store~$i\in S$ is characterised by four parameters
$(E_i, Q_i, P_i, \eta_i)$ as described below.  For each store
$i\in S$, we let $s_i(0)$ be the initial level of energy in store~$i$
and $s_i(t)$ be the level of energy in store~$i$ at (the end of) each
subsequent time~$t\ge1$.  Without loss of generality and for
simplicity of presentation of the necessary theory, we make the
convention that the level of energy in each store at any time is
measured by that volume of energy that it may ultimately supply, so
that, within the model, any \emph{(round-trip) inefficiency} of the
store is accounted for at the input stage.
While accounting
for such inefficiency is essential to our modelling and results, we
assume that energy, once stored, is not further subject to significant
time-dependent \emph{leakage}.
However, the theory of the present paper would require only minor
adjustments to incorporate such time-dependent leakage.

The successive levels of energy in each store~$i$ satisfy the
recursion
\begin{equation}
  \label{eq:1}
  s_i(t) = s_i(t-1) + r_i(t), \qquad t \ge 1,
\end{equation}
where $r_i(t)$ is the rate (positive or negative) at which energy is
added to the store~$i$ at the time~$t$.  Each store~$i\in S$ is
subject to \emph{capacity constraints}
\begin{equation}
  \label{eq:2}
  0 \le s_i(t) \le E_i,  \qquad t \ge 0,
\end{equation}
so that $E_i>0$ is the \emph{capacity} of store~$i$ (again measured by
the volume of energy it is capable of serving) and \emph{rate
  constraints}
\begin{equation}
  \label{eq:3}
  -P_i \le r_i(t) \le \eta_i Q_i,  \qquad t \ge 1.
\end{equation}
Here $P_i>0$ is the (maximum) \emph{output rate} of the store~$i$,
while $Q_i>0$ is the (maximum) rate at which externally available
energy may be used for \emph{input} to the store, with the resulting
rate at which the store fills being reduced by the \emph{round-trip
  efficiency}~$\eta_i$ of the store, where $0 < \eta_i \le 1$ (so that
the maximum rate at which usable energy may be added to the store is
$\eta_iQ_i$).  (For more general constraints, such as those imposed
by networks, see Section~\ref{sec:conclusions}).  Given the vector
$s(0) = (s_i(0),\,i\in S)$ of the initial levels of energy in the
stores, a \emph{policy} for the subsequent management of the stores is
a specification of the vector of rates~$r(t) = (r_i(t),\,i\in S)$, for
all times~$t\ge1$; equivalently, from~\eqref{eq:1}, it is a
specification of the vector of store levels
$s(t) = (s_i(t),\,i\in S)$, for all times~$t\ge1$.

The stores are used to manage as far as possible a \emph{residual
  energy} (surplus of generation over demand) \emph{process}
$(\re(t),\,t\ge1)$, where, for each time~$t$, a \emph{positive} value
of $\re(t)$ corresponds to surplus energy available for
\emph{charging} the stores, subject to losses due to inefficiency, and
a \emph{negative} value of $\re(t)$ corresponds to energy
\emph{demand} to be met as far as possible from the stores.
For any time~$t$, given the vector of
rates~$r(t) = (r_i(t),\,i\in S)$, define the \emph{imbalance}~$u(t)$
by
\begin{equation}
  \label{eq:4}
   u(t) = \re(t) - \left(\sum_{i\colon r_i(t) < 0} r_i(t)
   + \sum_{i\colon r_i(t) \ge 0} r_i(t)/\eta_i\right).
\end{equation}
The term in parentheses in~\eqref{eq:4} is the net rate at which
energy is input into the stores at time~$t$, as viewed externally,
i.e.\ before losses due to round-trip inefficiency.  We shall require
also that that the policy defined by the rate vectors~$r(t)$,
$t\ge 1$, is such that, at each successive time~$t$,
\begin{equation}
  \label{eq:5}
  \re(t) \ge 0 \quad \Rightarrow \quad u(t) \ge 0,
\end{equation}
so that, at any time~$t$ when there is an energy surplus
($\re(t) \ge 0$), the net energy input into the stores, as defined
above, cannot exceed that surplus; the quantity~$u(t)$ is then the
\emph{spilled energy} at time~$t$.  Similarly, we shall require that
\begin{equation}
  \label{eq:6}
  \re(t) \le 0 \quad \Rightarrow \quad u(t) \le 0,
\end{equation}
so that, at any time~$t$ when there is an energy shortfall
($\re(t) \le 0$), i.e.\ a positive net energy demand to be met from
stores, the net energy output of the the stores does not exceed that
demand;
the quantity~$-u(t)$ is then the \emph{unserved energy} at time~$t$.
(It is not difficult to see that, under any reasonable objective for
the use of the stores to manage the residual energy
process---including the minimisation of total unserved energy as
discussed below---there is nothing to be gained by overserving energy
at times~$t$ such that $\re(t) \le 0$.)
We shall say that a policy is \emph{feasible} for the management of
the stores if, for each~$t\ge1$, that policy satisfies the above
relations~\eqref{eq:1}--\eqref{eq:6}.

For any feasible policy, define the total unserved energy~$\U(t)$ to
any time~$t$ to be the sum of the unserved energies~$-u(t')$ at those
times $t'\le t$ such that $\re(t')\le 0$, i.e.,
\begin{equation}
  \label{eq:7}
  \U(t) =  -\sum_{t'\le t:\, \re(t') \le 0} u(t')
  = \sum_{t'\le t} \max(0, -u(t')),
\end{equation}
where the second equality in~\eqref{eq:7} above follows
from~\eqref{eq:5} and~\eqref{eq:6}.  Our objective is to determine a
feasible policy for the management of the stores so as to minimise the
total unserved energy over some given period of time.  It is possible
that, at any time~$t$, some store~$i$ may be \emph{charging}
($r_i(t) > 0$) while some other store~$j$ is simultaneously
\emph{discharging} ($r_j(t) < 0$).  We refer to this as
\emph{cross-charging}---although the model does not of course identify
the routes taken by individual electrons.  Although, in the presence
of storage inefficiencies, cross-charging is wasteful of energy, it is
nevertheless occasionally effective in enabling a better distribution
of energy among stores and avoiding the situation in which energy may
not be served at a sufficient rate because one or more stores are
empty.

We make also the following observation.  Suppose that some subset~$S'$
of the set of stores~$S$ is such that the stores $i\in S'$ have common
efficiencies $\eta_i$ and common capacity-to-power ratios $E_i/P_i$
and $E_i/Q_i$.  Then, clearly, these stores may be optimally managed
by keeping the fractional storage levels $s_i(t)/E_i$ equal across
$i\in S'$ and over all times~$t$, so that the stores in~$S'$
effectively behave as a single large store with total
capacity~$\sum_{i\in S'}E_i$ and total input and output rates
$\sum_{i\in S'}Q_i$ and $\sum_{i\in S'}P_i$ respectively.  (The reason
for this is that the single large store may notionally be partitioned
as the set $S'$ of smaller stores, and that there is then a one-one
correspondence between feasible policies using the former and those
using the latter.)  This is relevant when, as in the application of
Section~\ref{sec:appl-gb-energy}, we wish to consider the scheduling
and dimensioning of different storage technologies so as to obtain an
optimal mix of the latter.  Then, for this purpose, it is reasonable
to treat---to a good approximation---the storage to be provided by any
one technology as constituting a single large store.

\section{Nature of optimal policies}
\label{sec:nature-optim-solut}

We continue to take as our objective the minimisation of total
unserved energy over some given period of time.  We characterise
desirable properties of policies for the management of storage, and
show how at least approximately optimal policies may be determined.

In applications, the residual energy process to be managed is not
generally known in advance (so ruling out, e.g., the use of
straightforward linear programming approaches) and policies must be
chosen dynamically in response to evolving information about that
process.  Within our discrete-time setting, the information available
for decision-making at any time~$t$ will generally consist of the
vector of store levels $s(t-1) = (s_i(t-1),\,i\in S)$ at the end of
the preceding time period (equivalently the start of the time
period~$t$) together with the current value~$\re(t)$ of the residual
energy process.
However, this information may be supplemented by some, necessarily
probabilistic, prediction (however obtained) of the evolution of the
residual energy process subsequent to time~$t$.
We shall be particularly interested in identifying conditions under
which it is \emph{sufficient} to consider (feasible) policies in which
the decision to be made at any time~$t$, i.e.\ the choice of rates
vector~$r(t)$, depends \emph{only} on~$s(t-1)$ and~$\re(t)$, thereby
avoiding the need for real-time prediction of the future residual
energy process.  Such policies are usually referred to as
\emph{non-anticipatory}, or \emph{without foresight}.

Section~\ref{sec:greedy-policies} below defines \emph{greedy} policies
and shows that
it is sufficient to consider such policies.
Section~\ref{sec:non-pred-polic} discusses conditions under which a
(greedy) optimal, or approximately optimal, policy may be found from
within the class of non-anticipatory policies.
Section~\ref{sec:value-functions} shows that the immediate
optimisation problem to be solved at each successive time~$t$ may be
characterised as that of maximising a \emph{value function} defined on
the space of possible store (energy) levels, and identifies conditions
under which this latter problem is approximately a linear
programme---with a particularly simple, non-iterative, solution.

\subsection{Greedy policies}
\label{sec:greedy-policies}

We define a \emph{greedy} policy to be a feasible policy
in which, at each successive time~$t\ge 1$, and given the levels
$s(t-1)$ of the stores at the end of the preceding time period,
\begin{compactitem}[-]
\item if the residual energy $\re(t)\ge0$, i.e.\ there is energy
  available for charging the stores at time~$t$, then there is no
  possibility to increase any of the rates $r_i(t)$, $i\in S$ (without
  decreasing any of the others), and so further charge the stores,
  while keeping the policy feasible;
\item if the residual energy $\re(t)<0$, i.e.\ there is net energy
  demand at time~$t$, then there is no possibility to decrease any of
  the rates $r_i(t)$, $i\in S$ (without increasing any of the others),
  and so further serve demand, while keeping the policy feasible.
\end{compactitem}
Note that if $\re(t)=0$ at time~$t$, then, for a feasible policy, it
is necessarily the case, from~\eqref{eq:5} and~\eqref{eq:6}, that the
imbalance $u(t)=0$.

Proposition~\ref{proposition:greedy} and its corollary below
generalise a result of~\cite{ESDT} (in that case for a single store
which can only discharge). 

\begin{proposition}
  \label{proposition:greedy}
  Any feasible policy may be modified to be greedy while remaining
  feasible and while continuing to serve as least as much energy to
  each successive time~$t$.  Further, if the original policy is
  non-anticipatory, the modified policy may be taken to be
  non-anticipatory.
\end{proposition}

Proposition~\ref{proposition:greedy} is intuitively appealing: at
those times~$t$ such that $\re(t) \ge 0$, there is no point in
withholding energy which might be used for charging some store, since
the only possible ``benefit'' of doing so would be to allow further
energy---not exceeding the amount originally withheld---to be placed
in that store at a later time.  Similarly, at those times~$t$ such
that $\re(t) < 0$, there is no point in withholding energy in any
store which might be used to reduce unserved energy, since the only
possible ``benefit'' of doing so would be to allow additional
demand---not exceeding that originally withheld by that store---to be
met by that store at a later time.  A formal proof of
Proposition~\ref{proposition:greedy} is given in the Appendix.  Note
that greedy policies may involve cross-charging (see
Section~\ref{sec:model}).  Proposition~\ref{proposition:greedy} has
the following corollary.

\begin{corollary}
  \label{corollary:1}
  Suppose that the objective is the minimisation of unserved energy
  over some given period of time.  Then there is an optimal policy
  which is greedy.  Further, within the class of non-anticipatory
  policies there is a greedy policy which is optimal within this
  class.
\end{corollary}
We remark that under objectives other than the minimisation of total
unserved energy, optimal policies may fail to be greedy.  For example,
if unserved energy were costed nonlinearly, or differently at
different times, then at certain times it might be better to retain
stored energy for more profitable use at later times---see, for
example,~\cite{CZ2018}.


\subsection{Non-anticipatory policies}
\label{sec:non-pred-polic}


There are various conditions (see below) under which the optimal
policy may be taken to be not only greedy (see
Proposition~\ref{proposition:greedy}) but also \emph{non-anticipatory}
as defined above.  We are therefore led to consider whether it is
sufficient in applications to consider non-anticipatory policies---at
least to obtain results which are at least approximately optimal, and
to design and dimension storage configurations.  Two such
\emph{non-anticipatory} policies which work well under different
circumstances are:
\begin{compactitem}[-]
\item The \emph{greedy greatest-discharge-duration-first} (GGDDF)
  policy~(see \cite{Nash1978, ETA-pscc, CZ2018, ZTECA}) is a storage
  discharge policy for managing a given residual energy process which
  is negative (i.e.\ there is positive energy demand) over a given
  period of time, with the aim of minimising total unserved energy.
  It is defined by the requirement that, at each time~$t$, stores are
  prioritised for discharging in order of their \emph{residual
    discharge durations}, where the residual discharge duration of a
  store~$i$ at any time is defined as the energy in that store at the
  start of time divided by its maximum discharge rate~$P_i$.
  This non-anticipatory policy is designed to cope with rate
  constraints and to avoid as far as possible the situation in which
  there are times at which there is sufficient total stored energy,
  but this is located in too few stores.  It is optimal among policies
  which do not involve cross-charging, and more generally under the
  conditions discussed in~\cite{ZTECA}.  As also discussed there, it
  may be extended to situations where the residual energy process is
  both positive and negative. 
\item The \emph{greatest-round-trip-efficiency first} (GRTEF) policy
  is a greedy policy which is designed to cope with round-trip
  inefficiency: stores are both charged and discharged---in each case
  to the maximum feasible extent---in decreasing order of their
  efficiencies and no cross-charging takes place.  \emph{In the
    absence of output rate constraints}, the GRTEF policy may be shown
  to be optimal: straightforward coupling arguments, similar to those
  used to prove Proposition~\ref{proposition:greedy}, show that,
  amongst greedy policies, the GRTEF policy maximises the total stored
  energy $\sum_{i\in S} s_i(t)$ at any time, so that energy which may
  be served under any other policy may be served under this policy.  
\end{compactitem}

In practice, a reasonable and robust policy might be to use the GRTEF
policy whenever no store is close to empty, and otherwise to switch to
the GGDDF policy.  However, there is a need to find the right balance
between these two policies, and also to allow for the possibility of
cross-charging where this might be beneficial.
We therefore look more generally at non-anticipatory policies below.

\subsection{Value functions}
\label{sec:value-functions}

Standard dynamic programming theory (see, e.g.\ \cite{Bertsekas2012})
shows that, at any time~$t$, given a stochastic description of the
future evolution of the residual energy process, an optimal decision
at that time may be obtained through the computation of a \emph{value
  function} $V^t(s)$ defined on the set of possible
states~$s = (s_i,\, i\in S)$ of the stores, where
each~$s_i = s_i(t-1)$ is the level of energy in store~$i\in S$ at the
start of time~$t$.  The quantity~$V^t(s)$ may be interpreted as the
future \emph{value} of having the energy levels of the stores in
state~$s$ at time~$t$, relative to their being instead in some other
reference state, e.g.\ state~$0$, where \emph{value} is the negative
of \emph{cost} as measured by expected future unserved energy.  Then
the optimal decision at any time~$t$ is that which maximises the value
of the resulting state, less the cost of any energy unserved at
time~$t$.
In the present problem, such a stochastic description is generally
unavailable.  However, the value function might reasonably be
estimated from a sufficiently long residual energy data
series---typically of at least several years duration---especially if
one is able to assume (approximate) time-homogeneity of the above
stochastic description.  The latter assumption essentially
corresponds, over sufficiently long time periods, to the use of a
value function $V^t(s) = V(s)$ which is independent of time~$t$ and to
the use of a scheduling policy which is approximately non-anticipatory
(see below).

As previously indicated, we make the convention that the state~$s_i$
of each store~$i$ denotes the amount of energy which it is able to
serve---so that (in)efficiency losses are accounted for at the input
stage.  At any time~$t$, and given a stochastic description as above,
the value function~$V^t(s)$ may be computed in terms of absorption
probabilities (see, e.g.\ \cite{Durrett2019probability}).  For
each~$t$, let $v^t_i(s)$ be the partial derivative of the value
function $V^t(s)$ with respect to variation of the level~$s_i$ of each
store~$i\in S$.  Standard probabilistic coupling arguments, analogous
to those used to prove Proposition~\ref{proposition:greedy},
show that, for each $i\in S$, $v^t_i(s)$ lies between~$0$ and~$1$ and
is decreasing is $s_i$.  (For example, the positivity of $v^t_i(s)$
is simply the \emph{monotonicity} property that one can never be worse
off by having more stored energy---see Section~\ref{sec:model}---while
the inequality $v^t_i(s) \le 1$ reflects the fact that having one more
unit of energy in store~$i$ at time~$t$ can at most reduce future
unserved energy by a single unit.)  We assume that changes in store
energy levels are sufficiently small over each single time step~$t$
that changes to the value function may be measured using the above
partial derivatives.  Then the above problem of scheduling the
charging or discharging of the stores over the time step~$t$ becomes
that of choosing \emph{feasible} rates $r(t) = (r_i(t), \, i\in S)$ so
as to maximise
\begin{equation}
  \label{eq:8}
  \sum_{i\in S}v^t_i(s)r_i(t) - \max(0, -u(t)),
\end{equation}
where $s_i = s_i(t-1)$ is again the level of each store~$i\in S$ at
the end of the preceding time step~$t-1$.  This follows from the
characterisation of an optimal policy given at the start of this
section: the first term
in~\eqref{eq:8}) is the increase in the value function at time~$t$
corresponding to the choice of rates~$r(t)$, while the second term is
the unserved energy at that time (see Section~\ref{sec:model}).  
It follows from~\eqref{eq:8} and from the definition~\eqref{eq:4} of
$u(t)$ (coupled with the constraints~\eqref{eq:5} and~\eqref{eq:6})
that, under the above linearisation, the scheduling problem at each
time~$t$ becomes a linear programme.

When, at (the start of) any time~$t$, the state of the stores is given
by~$s = s(t-1)$, we shall say that any store~$i\in S$ has
\emph{charging priority} over any store~$j\in S$ if
$\eta_i v^t_i(s) > \eta_j v^t_j(s)$, and that any store~$i\in S$ has
\emph{discharging priority} over any store~$j\in S$ if
$v^t_i(s) < v^t_j(s)$.  Given the result~\eqref{eq:8},
Proposition~\ref{proposition:lp} below is again intuitively appealing;
we give a formal proof in the Appendix.

\begin{proposition}
  \label{proposition:lp}
  When the objective is the minimisation of total unserved energy over
  some given period of time, then, under the above linearisation, at
  each time~$t$ and with $s=s(t-1)$, the optimal charging, discharging
  and cross-charging decisions are given by the following procedure:
  \begin{compactitem}[-]
  \item \emph{when charging, i.e.\ if $\re(t) \ge 0$}, charge the
    stores in order of their charging priority, charging each
    successive store as far as permitted (the minimum of its input
    rate and its residual capacity) until the energy available for
    charging at time~$t$ is used as far as possible---any remaining
    energy being \emph{spilled};
  \item \emph{when discharging, i.e.\ if $\re(t) < 0$}, discharge the
    stores in order of their discharging priority, discharging each
    successive store as far as permitted (the minimum of its output
    rate and its available stored energy) until the demand at time~$t$
    is met as fully as possible---any remaining demand being
    \emph{unserved energy};
  \item \emph{subsequent to either of the above}, choose pairs of
    stores $(i,j)$ in succession by, at each successive stage,
    selecting store~$i$ to be the store with the highest discharging
    priority which is still able to supply energy at the time~$t$ and
    selecting store~$j$ to be the store with the highest charging
    priority which is still able to accept energy at the time~$t$; for
    each such successive pair $(i,j)$, provided that
    \begin{equation}
      \label{eq:9}
      v^t_i(s) < \eta_j v^t_j(s),
    \end{equation}
    \emph{cross-charge} as much energy as possible from store~$i$ to
    store~$j$.  Note that the above priorities are such that this
    process necessarily terminates on the first occasion such that the
    condition~\eqref{eq:9} fails to be satisfied, and further that no
    cross-charging can occur when $\re(t)\ge0$ and there is spilled
    energy, or when $\re(t)<0$ and there is unserved energy.
  \end{compactitem}
\end{proposition}

The pairing of stores for cross-charging in the above procedure is
entirely notional, and what is important is the policy thus defined.
However, when efficiencies are low, cross-charging occurs
infrequently.

In the examples of Section~\ref{sec:appl-gb-energy}, we consider
\emph{time-homogeneous} value function derivatives of the form
\begin{equation}
  \label{eq:10}
  v^t_i(s) = \exp (-\lambda_i s_i/P_i),
  \quad i \in S,
\end{equation}
essentially corresponding, as above, to the use of non-anticipatory
scheduling algorithms.  (However, data limitations---see the analysis
of Section~\ref{sec:appl-gb-energy}---mean that we use a single,
extremely long, residual energy dataset of 324,360 hourly observations
both to estimate the parameters~$\lambda_i$ and to examine the
effectiveness of the resulting policies.  Hence, the resulting
scheduling algorithms might be regarded as having, at each successive
point in time, some extremely mild anticipation of the future
evolution of the residual energy process.  Within the present
exploratory analysis this approach seems reasonable.)

The expression~\eqref{eq:10} is an approximation, both in its
assumption that, for each $i\in S$, the partial derivative~$v^t_i(s)$
depends only on the state~$s_i$ of store~$i$, and in the assumed
functional form of the dependence of~$v^t_i(s)$ on~$s_i$.  The former
assumption is equivalent to taking the value function as a sum of
separate contributions from each store (a reasonable first
approximation), while probabilistic large deviations
theory~\cite{Durrett2019probability} suggests that, under somewhat
idealised conditions, when the mean residual energy is positive,
the functions~$v^t_i(s)$ do decay exponentially.  However, we
primarily justify the use of the relation~\eqref{eq:10} in part by the
arguments below, and in part by its practical effectiveness---see the
examples of Section~\ref{sec:appl-gb-energy}.  Recall that what are
important are the induced decisions, as described above, on the
storage configuration space.  In particular, when the stores are under
pressure and hence discharging, it follows from the definition of
discharging priority above that it is only the ratios of the
parameters $\lambda_i$ which matter, except only for determining the
extent to which cross-charging should take place.  
Taking $\lambda_i = \lambda$ for all $i$ and for some $\lambda$
defines a policy which, when discharging, corresponds to the use of
the GGDDF policy, supplemented by a degree of cross-charging which
depends on the absolute value of the parameter~$\lambda$.  The
ability to further adjust the relative values of the
parameters~$\lambda_i$ between stores allows further tuning to
reflect their relative efficiencies; in particular, for a given volume
of stored energy, increasing the efficiency of a given store~$i\in S$
increases the desirability of having that energy stored in other
stores and reserving more of the capacity of store~$i$ for future
use---something which can be effected by increasing the
parameter~$\lambda_i$.

\section{Application to GB energy storage needs}
\label{sec:appl-gb-energy}

In this section we give an extended example of the application of the
preceding theory to the problem of dimensioning and scheduling future
GB energy storage needs within a net-zero environment.  Our primary
aim is to illustrate the practical applicability of the theory.
We also aim to show how, given also cost data, it might be used to
assist in storage dimensioning.  We are further concerned with the
extent to which it is sufficient to consider non-anticipatory
scheduling policies (those which do not assume foresight).  We explain
why one might expect to obtain similar conclusions for many other
countries.

A detailed description of the dimensioning problem,
together with details of all our storage, demand and renewable
generation data, including storage costs, is given
by~\cite{RoulstoneCosgrove2022}---work prepared in support of the
Royal Society report~\cite{RoyalSociety2023} on long-term large-scale
energy storage.  The paper~\cite{RoulstoneCosgrove2022} and the
companion paper~\cite{RoulstoneCosgrove2021} use a rather heuristic
scheduling algorithm, which occasionally leads to very high total
power requirements.  Additional discussion is given in the Royal
Society report itself, while the supplementary information for that
report~\cite{RoyalSociety2023_supp}, Section~3.3, discusses the
problem of sharing storage power requirements and compares in detail
the approach of~~\cite{RoulstoneCosgrove2022,RoulstoneCosgrove2021}
with that of the present paper---as also does the
paper~\cite{CRZ2023}.

We consider here a GB 2050 net-zero scenario, also considered
in~\cite{RoulstoneCosgrove2022, RoyalSociety2023}.  In this scenario
heating and transport are decarbonised, in line with the UK's 2050
net-zero commitment, thereby approximately doubling current
electricity demand to 600 TWh per year (see~\cite{UKCCC2019}), and all
electricity generation is renewable and provided by a mixture of 80\%
wind and 20\% solar generation.
We further assume a 30\% level of generation
overcapacity---corresponding to total renewable generation of 780 TWh
per year on average.  The above wind-solar mix and level of generation
overcapacity are those used in~\cite{RoulstoneCosgrove2022,
  RoyalSociety2023}, and are approximately optimal on the basis of the
generation and cost data considered there.  We also consider, very
briefly, the effect on storage dimensioning of a reduced level of
overcapacity of 25\%.

In the application of this section, we depart from our earlier
convention (made for mathematical simplicity) of
notionally accounting for all round-trip inefficiency at the input
stage.  We instead split the round-trip efficiency~$\eta_i$ of any
store~$i$ by taking both the input and output efficiencies to be given
by~$\eta_i^{0.5}$.  This revised convention increases both the
notional volumes of energy within any store~$i$ and the notional
capacity of the store~$i$ by a factor~$\eta_i^{-0.5}$.  This is in
line with most of the applied literature on energy storage needs and
makes our storage capacities below directly comparable with those
given elsewhere.


\paragraph{Generation and demand data.}

We use a dataset consisting of 37 years of hourly ``observations'' of
wind generation, solar generation and demand.  The wind and solar
generation data are both based on the 37-year (1980--2016) reanalysis
weather data of~\cite{StaffellPfenninger2016} together with assumed
installations of wind and solar farms distributed across GB and
appropriate to the above scenario, and with 80\% wind and 20\% solar
generation as above.  The derived generation data are scaled so as to
provide on average the required level of generation overcapacity
relative to the modelled demand.  The demand data are taken from a
year-long hourly demand profile again corresponding to the above 2050
scenario and in which there is 600 TWh of total demand; this profile
was prepared by Imperial College for the UK Committee on Climate
Change~\cite{UKCCC2019}.  As in \cite{RoulstoneCosgrove2022,
  RoyalSociety2023} this year-long set of hourly demand data has been
recycled to provide a 37-year trace to match the generation data.
(This is reasonable here as the between-years variability which may
present challenges to storage dimensioning and scheduling is likely to
arise primarily from the between-years variability in renewable
generation.  However, see also~\cite{GalloCassarino2018}.)
From these data we thus obtain a 37-year hourly \emph{residual energy}
(generation less demand) process to be managed by storage.  For the
chosen base level of 30\% generation overcapacity,
Figure~\ref{fig:residual_energy_30} shows a histogram and
autocorrelation function of the hourly residual energy process.  The
large variation in the residual energy process is to be compared with
the mean demand of 68.6 GW.

\begin{figure}[!ht]
  \centering
  \includegraphics[scale=0.6]{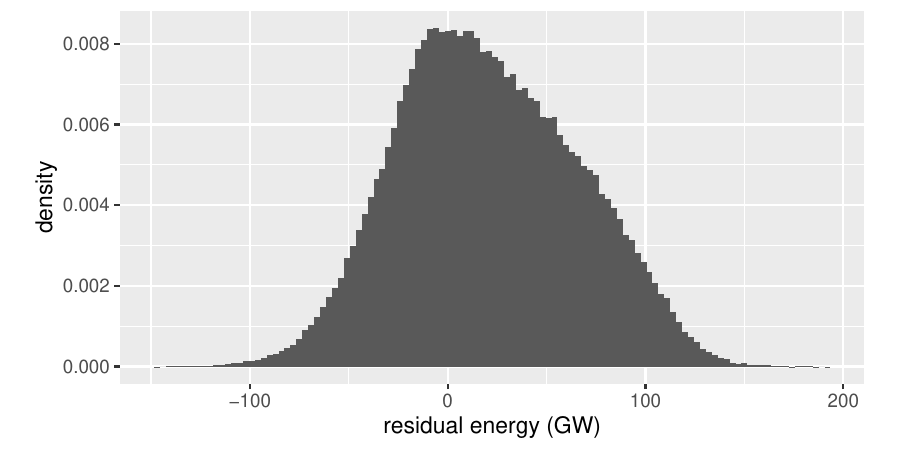}
  \includegraphics[scale=0.6]{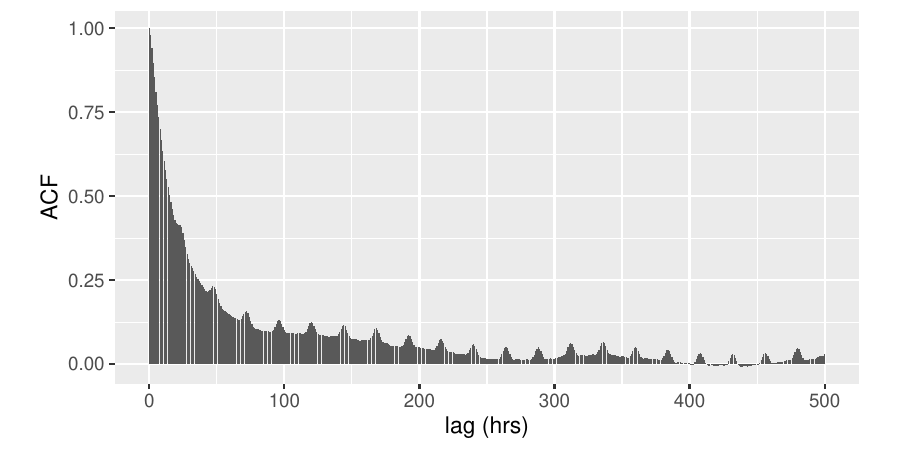}
  \caption{Histogram and autocorrelation function
    of hourly residual energy (30\% overcapacity).}
  \label{fig:residual_energy_30}
\end{figure}

In our examples below, the considered level of generation overcapacity
is 30\%.  However, it is useful to consider briefly the volume of
storage required to manage more general levels of overcapacity.  (Some
level of generation overcapacity is required, both to account for
losses due to inefficiencies in storage, and to keep the required
volume of storage within reasonable bounds.)  In particular, for a
\emph{single} store with given efficiency and without input or output
power constraints, there is a minimum store size and a minimum initial
store energy level such that the store can completely manage the above
residual energy process (i.e.\ with no unmet demand).
Figure~\ref{fig:st.size} plots, for various levels of store efficiency
and on the basis of our assumed 80\%--20\% wind-solar mix, this minimum
store size against the assumed level of overcapacity in the above
residual energy process.


\begin{figure}[ht]
  \centering
  \includegraphics[scale = 0.8]{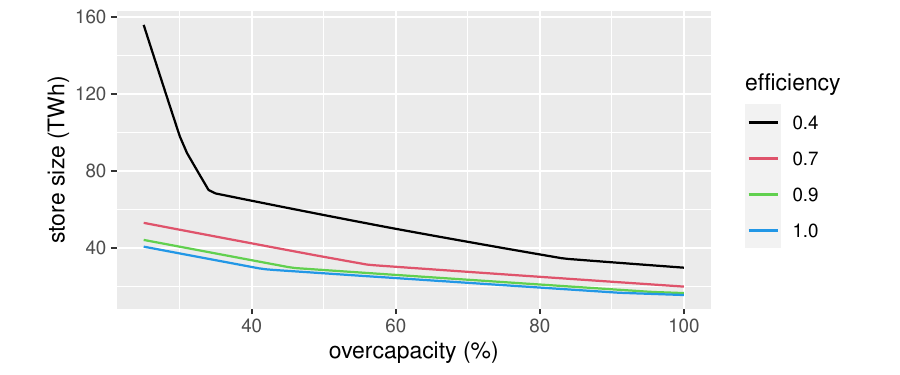}
  \caption{Dependence of minimal store size on level of generation
    overcapacity for various (round-trip) efficiencies.}
  \label{fig:st.size}
\end{figure}


\paragraph{Storage  data and costs.}

As discussed in Section~\ref{sec:introduction} and
in~\cite{RoulstoneCosgrove2022}, we consider three types of storage
with associated efficiencies:
\begin{compactitem}[-]
\item the \emph{short} store is intended primarily for the management
  of diurnal and other short-term variation, and has a low capacity
  requirement
  (see below); it is assumed that it can therefore use a technology
  such as Li-ion battery storage with a high efficiency, which we here
  take to be~$0.9$;
\item the \emph{medium} store is intended primarily for the management
  of weather-related variation on a timescale of days and weeks;
  it has very substantial capacity
  requirements and may require a technology such as ACAES which has a
  lower efficiency, which we here take to be $0.7$;
\item the \emph{long} store is intended for the management of
  seasonal and between-years variation (see
  Section~\ref{sec:introduction});
  it has an very high capacity requirement, and a power requirement
  which---on account of potentially high input/output costs---it is
  desirable to keep relatively modest; it requires a technology, such
  as hydrogen or similar chemical storage, which currently has a low
  efficiency, which we here take to be~$0.4$.
\end{compactitem}

We use storage costs from \cite{RoulstoneCosgrove2022}, Table 3, and
given in Table~\ref{tab:storage_costs} below (with storage capacity
measured according to to the convention of this section with regard to
accounting for inefficiency).  These costs are based on various recent
studies, as reported in~\cite{RoulstoneCosgrove2022}, and are
estimates of likely future storage costs in 2040 if the storage
technologies are applied on a large scale---current costs are
considerably higher.  For Li-ion batteries, the maximum input and
output rates are constrained to be the same, so that power costs may
be associated with input power.  However, there is huge uncertainty as
to future storage costs
(see~\cite{Cardenas2021,MIT2022,RoulstoneCosgrove2022,RoyalSociety2023}
for some discussion).

\begin{table}[ht]
  \centering
  \begin{tabular}{lrrr}
                      & capacity     & output power & input power\\
                      & (\$ per KWh) & (\$ per KW)  & (\$ per KW)\\
    \hline
    \emph{long}   (hydrogen) &        0.8   &         429   &        858\\
    \emph{medium} (ACAES)    &        9.0   &         200   &        200\\
    \emph{short}  (Li-ion)   &      100.0   &           0   &        180
  \end{tabular}
  \caption{Storage costs (US dollars) used for examples.}
  \label{tab:storage_costs}
\end{table}

Unit capacity costs decrease dramatically
as we move from the \emph{short}, to the \emph{medium}, to the
\emph{long} store, while unit power (rate) costs vary, again
considerably, in the opposite direction.
The aim in dimensioning and scheduling storage must therefore be to
arrive at a position in which the \emph{long} store is meeting most of
the total capacity requirement, while as much as is reasonably
possible of the total power requirement is being met by the
\emph{medium} and \emph{short} stores.

We treat GB as a single geographical node, ignoring possible network
constraints.  This is in line with most current studies of GB
\emph{long-term} storage needs, see,
e.g.~\cite{Cardenas2021,RoulstoneCosgrove2022,RoyalSociety2023}, and
with the annual Electricity Capacity Reports produced by the GB system
operator~\cite{NGECR}.  As at present, future network constraints are
unlikely to be continuously binding over periods of time in excess of
a few hours or a day or two at most, and
are primarily likely to affect short-term storage requirements.
However, see Section~\ref{sec:conclusions} for how such constraints
could be included in the present approach.



We take the reliability standard to be given by 24 GWh per year
unserved energy and optimise scheduling and dimensioning subject to to
constraint that this standard is met.  This results in an average
\emph{number of hours} per year in which there is unserved energy
which is roughly in line with the current GB standard of a maximum of
3 such hours per year.  However, modest variation of the chosen
reliability standard
makes very little difference to our conclusions.

Example~\ref{ex:1} below considers a single store.  In the remaining
examples we \emph{schedule} storage using time-homogeneous value
function derivatives $v_i(s)$ given by~\eqref{eq:10}---with~$s$ defined
as there to be the volume of stored energy which may be output, and with
the parameters~$\lambda_i$, $i\in S$, estimated from the data as
described above.  Thus, as previously discussed, the scheduling is
almost completely non-anticipatory.  
We consider also the optimality of the scheduling algorithms used.

We take the stores to be initially full.  However, 
in all our examples, stores fill rapidly regardless of their initial
energy levels, and these levels are in general independent of their
initial values by the end of the first year of the 37-year period
considered.

\begin{example}
  \label{ex:1}
  \emph{Single long (hydrogen) store with efficiency 0.4.}  We first
  consider the management of the residual energy process by a single
  store, optimally dimensioned with respect to cost.  If a single
  store is to be used, then, of the technologies considered here and
  on the basis of the present, \emph{as yet very uncertain}, costs, a
  hydrogen store is the only economic possibility---see
  also~\cite{RoulstoneCosgrove2022,RoyalSociety2023}.

  The unserved energy is clearly a decreasing function of each of the
  store capacity~$E$, the maximum input power~$Q$ and the maximum
  output power~$P$.  For any given value of~$P$, we may thus easily
  minimise the overall cost over $(E, Q)$.  It then turns
  out---unsurprisingly given the stringent reliability standard---that
  the overall cost is here minimised by taking $P$ to be the minimum
  possible value (115.9 GW at the assumed 30\% generation
  overcapacity) such that the given reliability standard of 24 GWh
  unserved energy per year is satisfied.  Table~\ref{tab:sl} shows the
  optimal storage dimensions and associated costs.
  This store capacity is larger than that suggested by
  Figure~\ref{fig:st.size}, where the maximum store input power $Q$
  was unconstrained: on the basis of the present costs, it is more
  economic to reduce $Q$ at the expense of allowing the store
  capacity~$E$ to increase.

  \begin{table}[ht]
    \centering
    \begin{tabular}{rrrrr}
                   & capacity      & output power & input power & total\\
      \hline
      size         & 120.4 TWh     & 115.9 GW     & 80.0 GW     &   \\
      cost (\$ bn) &  96.3 ~~~~~~~ &  49.7 ~~~~~  & 68.6 ~~~~~  & 214.7\\
      \hline                            
    \end{tabular}

   
    \caption{Single \emph{long} (hydrogen) store: dimensions and costs.}
    \label{tab:sl}
  \end{table}

  At a lower level of 25\% generation overcapacity (and at the same
  reliability standard) the total cost of hydrogen storage as above is
  \$257.5 bn.  This is \$42.8 bn greater than that 30\% overcapacity,
  making the 30\% level of overcapacity more economic on the basis of
  the storage and generation costs given
  by~\cite{RoulstoneCosgrove2022}.

  For 30\% generation overcapacity, Figure~\ref{fig:sl.eu} plots
  cumulative unserved energy against time.
  The store never completely empties and so unserved energy occurs
  only at those times at which the output power $P$ of the store is
  insufficient to serve demand.

  \begin{figure}[ht]
    \centering
    \includegraphics[scale=0.7]{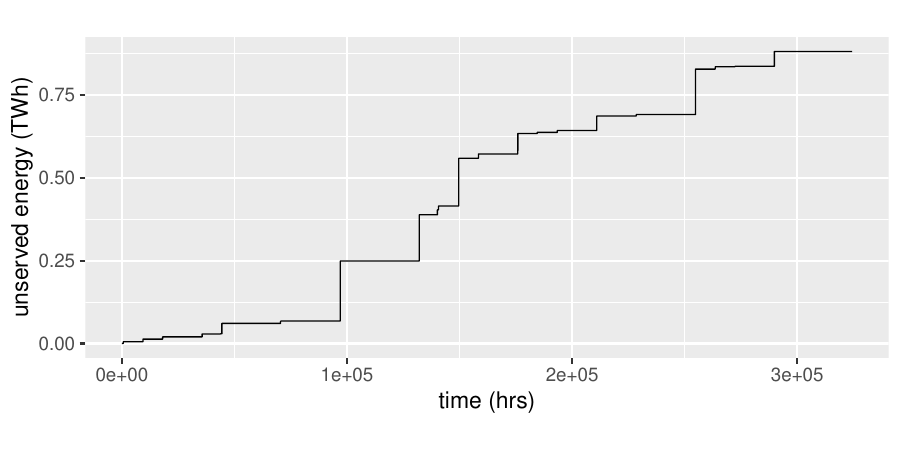}
    \caption{Example~\ref{ex:1}: single \emph{long} (hydrogen) store:
      cumulative unserved energy.}
    \label{fig:sl.eu}
  \end{figure}

  Figure~\ref{fig:sl.sl} shows the corresponding processes formed by
  the successive energy levels within the store.  A substantial
  fraction of the store capacity is needed solely to manage the single
  period of large shortfall in the residual energy process occurring
  at around 275,000 hours into the 37-yr (324,360 hour) period
  studied.  This underlines the importance of using a residual energy
  time-series which is sufficiently long to capture those events such
  as sustained wind droughts which only occur perhaps once every few
  decades---see also~\cite{RoyalSociety2023}.

  \begin{figure}[ht]
    \centering
    \includegraphics[scale=0.7]{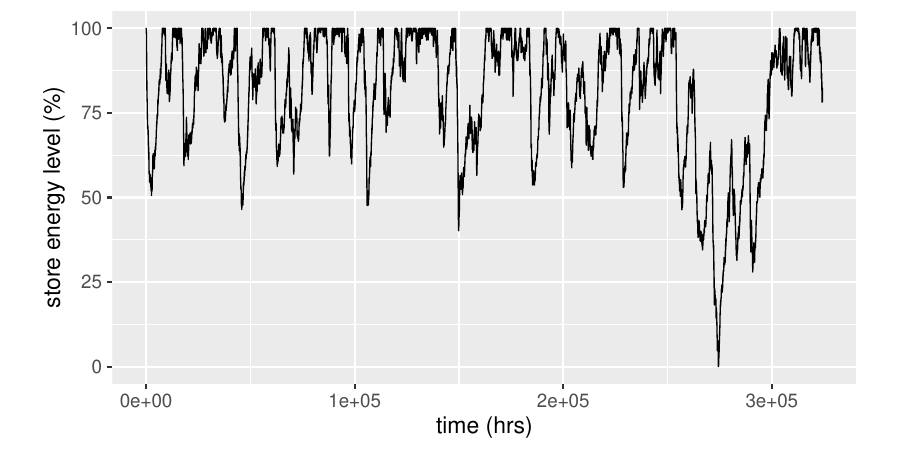}
    \caption{Example~\ref{ex:1}: single \emph{long} (hydrogen) store:
       successive store energy levels.}
    \label{fig:sl.sl}
  \end{figure}
  

\end{example}

\begin{example}
  \label{ex:2}
  \emph{Long (hydrogen) store with efficiency 0.4 plus medium (ACAES)
    store with efficiency 0.7.}  In this example we show that, again
  on the basis of the cost data used here and the considered level of
  generation overcapacity, extremely large savings (of the order of
  tens of billions of dollars) are to be made by the use of a suitable
  mixture of storage technologies.

  We choose \emph{medium} (ACAES) store dimensions as below:
  some numerical experimentation shows these to be at least close to
  optimal with respect to overall cost minimisation.  Then, given
  these \emph{medium} store dimensions, and subject to the given
  reliability standard of 24 GWh unserved energy per year, the
  \emph{long} (hydrogen) store may be optimally dimensioned---given
  the use of the value-function based scheduling algorithm, and again
  to a very good approximation---as previously.  Table~\ref{tab:slm}
  shows the optimal storage dimensions and associated costs (again for
  the assumed 30\% generation overcapacity).  Note that
  the combined output power of the two stores is only slightly greater
  than that of the single store of Example~\ref{ex:1}, so that the two
  stores are effectively cooperating in meeting the total power
  requirement.
  
  \begin{table}[ht]
    \centering

    
    \begin{tabular}{l|r|rrrr}
      \multicolumn{2}{c}{}
                  & capacity  & output power & input power & total cost\\[2pt]
      \hline
      \emph{long} & size   &  72.8 TWh     &  96.2 GW     & 53.3 GW    &   \\
      store & cost (\$ bn) &  58.2 ~~~~~~~ &  41.3 ~~~~~  & 45.7 ~~~~~ & 145.2\\
      \hline
       \emph{medium} & size &  2.5 TWh     &  21.0 GW     & 21.1 GW    &  \\
      store & cost (\$ bn) &  22.5 ~~~~~~~ &   4.2 ~~~~~  &  4.2 ~~~~~ & 30.9\\
      \hline
      \multicolumn{2}{r}{Total cost (\$ bn)} &&&& \rule{0pt}{3ex} 176.2
    \end{tabular}
    
   

    
    \caption{\emph{Long} (hydrogen) store plus \emph{medium} (ACAES)
      store: dimensions and costs.}
    \label{tab:slm}
  \end{table}

  The reason for the very large costs savings of \$38.5 bn, relative
  to the use of a single storage technology,
  is as follows.  The low efficiency (0.4) of the \emph{long}
  hydrogen store means that, when used on its own, its capacity is
  necessarily much greater than would have been the case had its
  efficiency been higher---see also Figure~\ref{fig:st.size}.  The
  greater efficiency (0.7) of the very much smaller \emph{medium}
  ACAES store introduced in this example allows it to be used to
  cycle rapidly---see Figure~\ref{fig:slm.sl}---serving a
  disproportionate share of the demand in relation to its capacity,
  and thereby \emph{greatly} reducing the capacity requirement for the
  \emph{long} store.  At lower levels of generation overcapacity,
  storage efficiency becomes even more important (for example, at 25\%
  overcapacity, cost savings of \$53.3 bn may achieved by the
  introduction of the \emph{medium} ACAES store).  We observe also
  that above explanation for the large cost savings to be achieved by
  the use of a mix of technologies, relative to the use of either on
  its own, is equally applicable to other systems where there is
  variation on multiple timescales.

  The parameters $\lambda_i$ of the scheduling algorithm
  (equation~\eqref{eq:10}) are given by
  $(\lambda_l, \lambda_m) = (0.0011, 0.01)$ per hour.
  The annual unserved energy just meets the required reliability
  standard.  The average annual volumes of
  energy served externally, i.e.\ to meet demand, by the \emph{long}
  and \emph{medium} stores are 47.6 TWh and 35.9 TWh
  respectively---with,
  in this example, negligible extra energy being used for
  cross-charging.  Thus
  the much smaller \emph{medium} store serves a comparable volume of
  energy to the \emph{long} store.

  Figure~\ref{fig:slm.eu} plots cumulative unserved energy (here
  averaging 23.9 GWh per year) against time, together with the
  corresponding process in which there is only unserved energy to the
  extent that demand exceeds the combined output power (117.2 GW) of
  the two stores; this latter process provides a lower bound (20.4
  GWh per year or 754 GWh over the entire 37-year period) on the
  unserved energy achievable.
  There is thus only one significant occasion (at around 150,000
  hours) on which, for the original fully constrained storage system,
  there is unserved energy over and above that forced by the power
  constraint; this is the result of the \emph{medium} store emptying
  and the \emph{long} store then being unable on its own to serve
  energy at the required rate.  The question now arises as to whether
  different (anticipatory) management of the stores, in the period
  immediately preceding this occasion, could have avoided this.  The
  linear programming solution to the unserved-energy minimisation
  problem defined in Section~\ref{sec:model} does, in this example,
  find such a policy, but this solution requires advance knowledge
  (foresight) of the residual energy process over the entire 37-year
  time period considered, and so does not provide a realistic
  practical approach.  However, it is clear that, under the
  essentially non-anticipatory policy found by the present algorithm,
  the stores are very close to being optimally controlled.

  \begin{figure}[!ht]
    \centering
    \includegraphics[scale = 0.8]{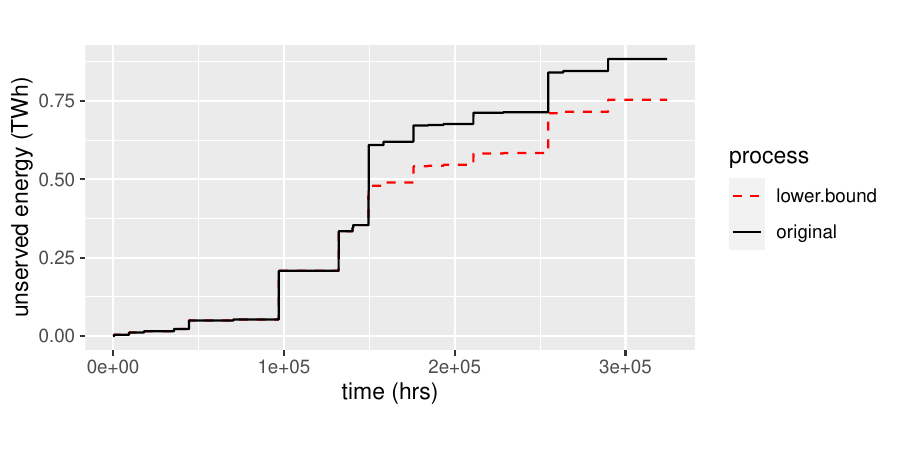}
    \vspace{-2ex}
    \caption{Example~\ref{ex:2}: plot of cumulative unserved energy
      against time (black) together with lower bounding process
      (red).}
    \label{fig:slm.eu}
  \end{figure}
 
  Figure~\ref{fig:slm.sl} plots the percentage levels of energy in
  store during a two-year period, starting at time 265,000 hours and
  surrounding the one point in time at which the \emph{long} store
  comes very close to emptying.
  The \emph{medium} store cycles rapidly, thereby using its higher
  efficiency to greatly reduce the capacity and input rate
  requirements on the \emph{long} store.  It nevertheless generally
  reserves about half its capacity so that it is available to assist
  in any ``emergency'' in which the demand exceeds the output power of
  the \emph{long} store alone.  The exception to this occurs at those
  times when the \emph{long} store is itself close to emptying, and
  when the \emph{medium} store must therefore work harder to further
  relieve the pressure on the \emph{long} store.

  \begin{figure}[!ht]
    \centering
    \includegraphics[scale = 0.8]{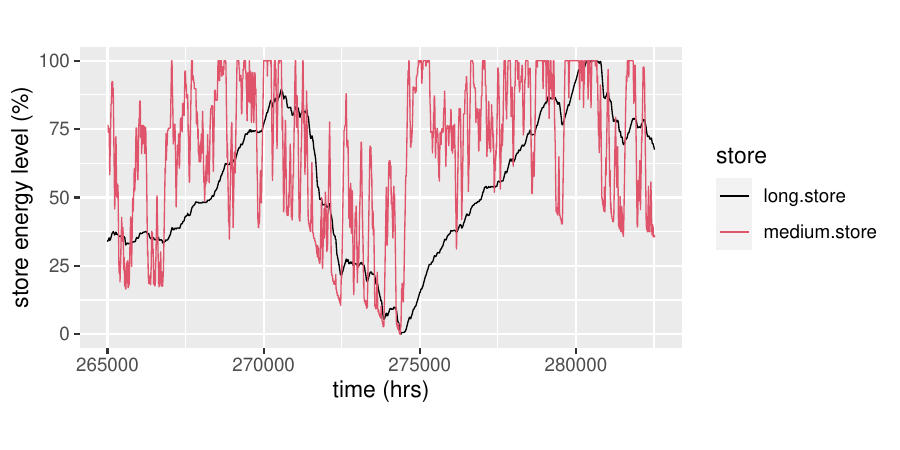}
    \vspace{-2ex}
    \caption{Example~\ref{ex:2}, 30\% generation overcapacity: plot of
      store levels (\%) against time.}
    \label{fig:slm.sl}
  \end{figure}
  
  The present example may also be used to further show the importance
  of \emph{not assuming foresight}---i.e.\ the importance of using
  non-anticipatory policies---in the \emph{dimensioning} of multiple
  storage types.  If foresight \emph{were} assumed, so that the
  scheduling could be done using a linear programming approach as
  above, then it turns out that it would be possible to reduce the
  capacity of the \emph{medium} store from 2.5 TWh to 1.29 TWh (at a
  cost saving of \$10.9 bn) while still continuing to meet all demand
  except that which is in excess of the combined output power (117.2
  GW) of the two stores.  In particular, the chosen reliability
  standard would again be comfortably met.  The reason why, under the
  assumption of foresight, the capacity of the \emph{medium} store may
  be nearly halved, is that the \emph{medium} store may then use
  \emph{nearly all} its capacity for cycling to reduce the input and
  capacity requirements on the \emph{long} store; on the rare
  occasions when it is anticipated that the power output capability of
  the \emph{long} store will need to be supplemented, the
  \emph{medium} store may reduce its cycling sufficiently far in
  advance so as to hold the necessary capacity in reserve.
\end{example}

\begin{example}
  \label{ex:3}
  \emph{Long (hydrogen) store with efficiency~0.4 plus short (Li-ion)
    store with efficiency~0.9.}  In the context of long-term GB
  storage needs, a necessarily relatively small \emph{short} store
  (Li-ion battery) can probably only make a relatively modest
  contribution.  Analogously to Example~\ref{ex:2}, we here explore
  the extent to which it is possible for it to assist in the provision
  of storage mostly provided by a \emph{long} (hydrogen) store.

  As in Example~\ref{ex:2}, we choose \emph{short} (Li-ion) store
  dimensions which (with some experimentation) appear to work well
  with respect to overall cost minimisation, subject here to equal
  input and output power ratings---see the discussion above.
  Given the \emph{short} store dimensions, and subject to the given
  reliability standard of 24 GWh per year, the \emph{long} (hydrogen)
  store may again be optimally dimensioned as in Example~\ref{ex:2}.


  Table~\ref{tab:sls} shows storage dimensions and associated costs
  (for 30\% generation overcapacity).  These results are to be
  compared with those of Table~\ref{tab:sl}.  What is remarkable is
  that
  a very large reduction in the capacity of the \emph{long} (hydrogen)
  store is achieved through the introduction of a \emph{short}
  (Li-ion) store of \emph{very} small capacity.  This is again
  primarily achieved through constant rapid cycling by the
  \emph{short} store so as to exploit its much greater
  efficiency---see Figure~\ref{fig:sls.sl} below.  The total cost
  saving of \$6.3 bn is similarly noteworthy.

  \begin{table}[ht]
    \centering

    
    \begin{tabular}{l|r|rrrr}
      \multicolumn{2}{c}{}
                  & capacity  & output power & input power & total cost\\[2pt]
      \hline
      \emph{long} & size   & 101.2~~~ TWh     & 115.9 GW     & 77.5 GW    &   \\
      store & cost (\$ bn) &  81.0 ~~~~~~~~~~ &  49.7 ~~~~~  & 66.5 ~~~~~ & 197.2\\
      \hline
       \emph{short} & size &  0.085 TWh      &  15.0 GW     & 15.0 GW    &  \\
      store & cost (\$ bn) &  8.5 ~~~~~~~~~~ &   0.0 ~~~~~  &  2.7 ~~~~~ & 11.2\\
      \hline
      \multicolumn{2}{r}{Total cost (\$ bn)} &&&& \rule{0pt}{3ex} 208.4
    \end{tabular}
    
   

    
    \caption{\emph{Long} (hydrogen) store plus \emph{short} (Li-ion)
      store: dimensions and costs.}
    \label{tab:sls}
  \end{table}

  The parameters $\lambda_i$ of the scheduling algorithm
  (equation~\eqref{eq:10}) are given by
  $(\lambda_l, \lambda_s) = (0.000001, 0.1)$ per hour.
  The average annual volumes of energy served externally by the
  \emph{long} and \emph{short} stores are 73.8 TWh and 9.8 TWh
  respectively---again with the given reliability standard just being
  met and
  with negligible extra energy being used for cross-charging.  The
  linear programming solution to the unserved-energy minimisation
  problem defined in Section~\ref{sec:model} finds an absolute minimum
  unserved energy of 21.4 GWh per year, or 791 GWh over the entire
  37-year period, but again this solution requires advance knowledge
  of the residual energy process over all time.  Thus again the
  essentially non-anticipatory policy of the present algorithm finds a
  control which is very close to optimal.

  Figure~\ref{fig:sls.sl} plots the percentage levels of energy in
  store during the same two-year period considered in
  Example~\ref{ex:2}.
  It is seen that the \emph{short} (Li-ion) store here
  devotes \emph{all} its capacity  to cycling rapidly,
  using its higher efficiency to \emph{greatly} reduce the capacity
  and, to a lesser extent, the input power requirements for the
  \emph{long} store.  The capacity costs of the \emph{short} store are
  such that it is not worth further increasing its capacity so as to
  reserve energy to enable the reduction of the \emph{output power}
  requirement of the \emph{long} store.  Hence the pattern of usage of
  the \emph{short} store is here different from that of the
  \emph{medium} store in the previous example.

  \begin{figure}[ht]
    \centering
    \includegraphics[scale = 0.8]{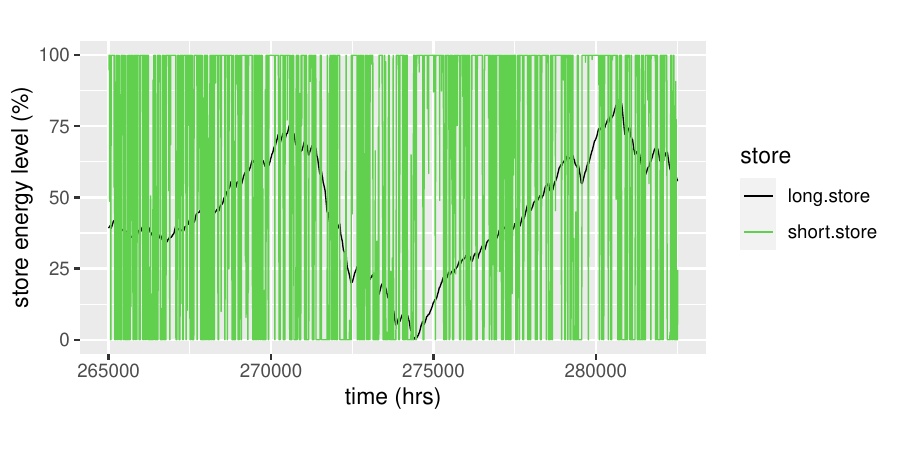}
    \vspace{-3ex}
    \caption{Example~\ref{ex:3}, 30\% generation overcapacity: plot of
      store levels (\%) against time.}
    \label{fig:sls.sl}
  \end{figure}
  
\end{example}

\begin{example}
  \label{ex:4}
  \emph{Long (hydrogen) store with efficiency~0.4 plus medium (ACAES)
    store with efficiency~0.7 plus short (Li-ion) store with
    efficiency~0.9.}
  
  In this final example we take the set-up of Example~\ref{ex:2},
  i.e.\ \emph{long} (hydrogen) store plus \emph{medium} (ACAES) store,
  and consider whether any further overall cost reduction can be
  obtained by the addition of a \emph{short} (Li-ion) store.
  For the assumed 30\% generation overcapacity and the given
  reliability standard, some experimentation shows that the storage
  dimensions and associated costs given in Table~\ref{tab:slms} are at
  least approximately optimal and lead to a modest cost
  reduction---relative to Example~\ref{ex:2}---of \$0.34 bn.
  Here the \emph{short} store is relatively very small indeed;
  however, variation of its dimensions does not seem to assist in
  further reducing overall costs.  Thus, of our four examples and on
  the basis of the present costs, the present three-store mix appears
  to be the most economical.

    \begin{table}[ht]
    \centering

    
    \begin{tabular}{l|r|rrrr}
      \multicolumn{2}{c}{}
                  & capacity  & output power & input power & total cost\\[2pt]
      \hline
      \emph{long} & size   &  72.2~~ TWh     &  96.2 GW     & 53.3 GW    &   \\
      store & cost (\$ bn) &  57.8 ~~~~~~~~~ &  41.3 ~~~~~  & 45.7 ~~~~~ & 144.8\\
      \hline
       \emph{medium} & size &  2.44~ TWh     &  21.0 GW     & 21.1 GW    &  \\
      store & cost (\$ bn) &  22.0 ~~~~~~~~~ &   4.2 ~~~~~  &  4.2 ~~~~~ & 30.4\\
      \hline
       \emph{short} & size &  0.005 TWh      &   2.0 GW     &  2.0 GW    &  \\
      store & cost (\$ bn) &  0.5 ~~~~~~~~~~ &   0.0 ~~~~~  &  0.2 ~~~~~ & 0.7\\
      \hline
      \multicolumn{2}{r}{Total cost (\$ bn)} &&&& \rule{0pt}{3ex} 175.8
    \end{tabular}
    
   

    
    \caption{\emph{Long} (hydrogen) store plus \emph{medium} (ACAES)
      store plus \emph{short} (Li-ion) store : dimensions and costs.}
    \label{tab:slms}
  \end{table}

  The parameters $\lambda_i$ of the scheduling algorithm
  (equation~\eqref{eq:10}) are given by
  $(\lambda_l, \lambda_m, \lambda_s) = (0.001, 0.011, 0.035)$ per hour.
  The annual volumes of energy served externally by the \emph{long},
  \emph{medium} and \emph{short} stores are 47.2 TWh, 36.4 TWh and
  0.012 TWh respectively,
  again with negligible extra energy being used for cross-charging.
  The linear programming solution to the unserved-energy minimisation
  problem defined in Section~\ref{sec:model} finds an absolute minimum
  of 17.9 GWh unserved energy per year, or 664 GWh over the entire
  37-year period,
  so that, as in previous examples, the non-anticipatory policy of the
  present algorithm finds a control which is reasonably close to
  optimal.
  

  Figure~\ref{fig:slms.sl} plots the percentage levels of energy in
  store during the same two-year period considered in
  Examples~\ref{ex:2} and~\ref{ex:3}.  The behaviour of the
  \emph{long} and \emph{medium} store processes is, unsurprisingly,
  essentially as in Example~\ref{ex:2} (Figure~\ref{fig:slm.sl}).  The
  behaviour of the \emph{short} store is here interesting.  For most
  of the time it remains full, reserving its energy for those
  occasions on which it may be called on to act in an emergency.
  However, as the \emph{long} and \emph{medium} stores come close to
  being empty, the \emph{short} store cycles as rapidly as
  possible---essentially in an attempt to prevent the former two
  stores actually emptying.

  \begin{figure}[!ht]
    \centering
    \includegraphics[scale = 0.8]{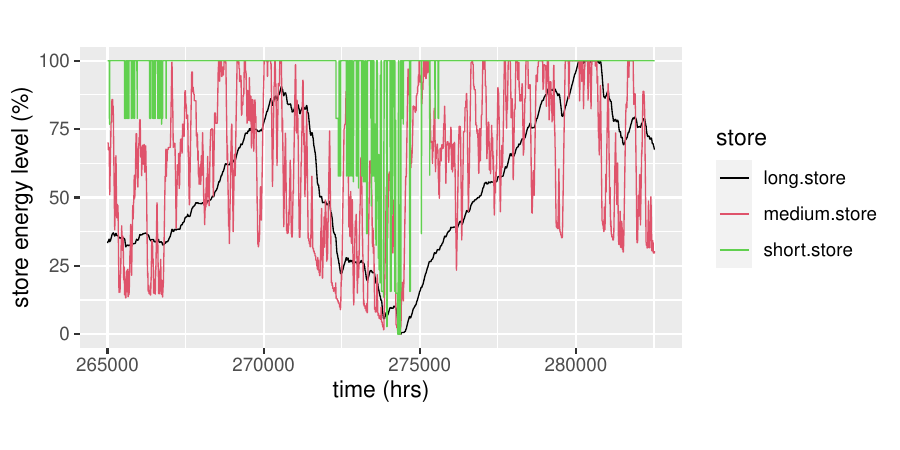}
    \vspace{-2ex}
    \caption{Example~\ref{ex:4}, 30\% generation overcapacity: plot of
      store levels (\%) against time.}
    \label{fig:slms.sl}
  \end{figure}

\end{example}

\section{Conclusions}
\label{sec:conclusions}

Future electricity systems may well require extremely high volumes of
energy storage with a mixture of storage technologies.  This paper has
studied the societal problems of scheduling and dimensioning such
storage, with the scheduling objective of minimising total unserved
energy over time, and the dimensioning objective of doing as
economically as possible.  We have identified properties of optimal
scheduling policies and have argued that a value-function (dynamic
programming) based approach is theoretically optimal.  We have further
shown that the optimal scheduling problem to be solved at each
successive point in time reduces, to a good approximation, to a linear
programme with a particularly simple solution.

We have been particularly concerned to develop \emph{non-anticipatory}
scheduling policies---i.e. policies which do not require the use of
\emph{foresight}---which are robust and suitable for real-time
implementation, and have demonstrated their success in practical
application.  Such policies also permit scheduling over arbitrarily
long periods of time without undue numerical complexity.  However,
there are very occasional situations in which a reliable forecast of,
for example, a prolonged energy drought would make it sensible to
modify these scheduling policies so as to maximally conserve energy.

We have considered the practical application of the above theory to
future GB energy storage needs, and shown, informally, how it may be
used for the dimensioning of heterogeneous storage technologies.
Notably, we have shown that the joint management of such technologies
may greatly reduce overall costs (though the latter are as yet very
uncertain), and we have indicated why similar very large savings are
to be expected in other systems.

We have not formally considered the modelling and analysis of network
constraints.  To do so, it would be necessary to identify storage
locations with respect to the network.  The effect of such constraints
on the model of the present paper would be to add further linear
constraints (in addition to~\eqref{eq:3}) on the input and output
rates of the stores.  Proposition~\ref{proposition:greedy} would
continue to hold, with obvious modifications to the proof.  Further
the general theory given in Section~\ref{sec:nature-optim-solut}, in
particular the value-function based approach to optimal scheduling
would continue to be applicable---with some modification required to
Proposition~\ref{proposition:lp}.

Nor have we considered how to effect such storage dimensioning and
management within a \emph{market} environment in which storage is
privately owned and operated by players each seeking to optimise their
own returns.  It seems likely that, under such circumstances, the
effective use of storage would require management over extended
periods of time by the electricity system operator and that
contractual arrangements, including the possible introduction of
storage capacity markets, would have to be such as to make this
possible (see~\cite{ZWD2022} for how this might be done).

\backmatter

\bmhead{Acknowledgments} The author gratefully acknowledges support by
Towards Turing 2.0 under the EPSRC Grant EP/W037211/1 and by the Alan
Turing Institute.  The author would also like to thank the Isaac
Newton Institute for Mathematical Sciences for support during the
programme Mathematics of Energy Systems
(\href{https://www.newton.ac.uk/event/mes/}{https://www.newton.ac.uk/event/mes/}),
when early work on this paper was undertaken.  He is grateful to Tony
Roulstone and Paul Cosgrove of the University of Cambridge for many
helpful discussions on GB storage needs and for making available all
the data used in the analysis of this paper.  He is similarly grateful
to Chris Dent of the University of Edinburgh for assistance with the
linear programming computations.  He also thanks many other
colleagues, notably Simon Tindemans of Delft University of Technology,
Frank Kelly of the University of Cambridge and James Cruise, for wider
discussions on the management of energy storage.  Finally, he is most
grateful to the reviewers for many insightful comments and suggestions
for improvements.

\bmhead{Conflict of interest}
The author declares that he has no conflict of interest.

\begin{appendices}

\section{Proofs}

\begin{proof}[Proof of Proposition~\ref{proposition:greedy}]
  Given any feasible policy, we show how, for each successive
  time~$t$, the policy may be modified at each time $t'\ge t$ in such
  a way that the policy becomes greedy at the time~$t$ and remains
  feasible at at each time $t'\ge t$ (as well as at times prior to
  $t$), and further continues to serve at least as much energy in
  total to each successive time.  Iterative application of this
  procedure over successive times~$t$ then finally yields a policy
  which is feasible and greedy at all times and which continues to
  serve at least as much energy in total to each successive time.
  (Thus, at any time~$t'$, the \emph{final} modification to the
  original policy is obtained by a succession of the above
  modifications associated with the successive times~$t\le t'$.)

  Suppose that, immediately prior and immediately subsequent to the
  modification associated with the time~$t$ (which affects the storage
  rates and levels for those times $t'\ge t$), the storage rates are
  defined, for each time~$t'$, respectively by
  $r(t') = (r_i(t'),\,i\in S)$ and
  $\hat r(t') = (\hat r_i(t'),\,i\in S)$, with the corresponding store
  levels being given respectively by $s(t') = (s_i(t'),\,i\in S)$ and
  $\hat s(t') = (\hat s_i(t'),\,i\in S)$, and with the \emph{total}
  unserved energy to each successive time $t'\ge t$ being given
  respectively by $\U(t)$ and $\hat{\U}(t')$ as defined
  by~\eqref{eq:7}.  Then the modification associated with the time~$t$
  is defined as follows.
  \begin{compactenum}[1.]
  \item If $\re(t) \ge 0$, increase (if necessary) the rates
    $(r_i(t),\,i\in S)$, at which energy is supplied to the stores at
    time~$t$ to $(\hat r_i(t),\,i\in S)$, so that the policy becomes
    greedy at time~$t$ while remaining feasible at that time.  Note
    that the effect of this is to increase (weakly) the store levels
    at time~$t$ so that $\hat s_i(t) \ge s_i(t)$, $i\in S$.  For
    times~$t' > t$ and for each $i\in S$, set
    $\hat r_i(t') = \min(r_i(t'), E_i - s_i(t'-1))$.  Then the
    modified policy remains feasible and it is clear, by induction,
    that $\hat s_i(t') \ge s_i(t')$ for all $i\in S$ and for all
    $t'\ge t$.  Further, since $\re(t) \ge 0$ there is no unserved
    energy at time~$t$ and since, for $t'>t$ such that $\re(t')<0$, we
    have $\hat r_i(t') \le r_i(t)$, $i\in S$ (implying,
    from~\eqref{eq:4}, that the unserved energy $-u(t')$ does not
    increase) it follows that the total unserved energy to each
    successive time~$t'\ge t$ does not increase.
  \item If $\re(t) < 0$, reduce (if necessary) the rates
    $(r_i(t),\,i\in S)$, to $(\hat r_i(t),\,i\in S)$, so that the
    policy becomes greedy at time~$t$ while remaining feasible at that
    time.  For times~$t' > t$ and for each $i\in S$, set

    \begin{equation}
      \label{eq:11}
      \hat r_i(t') = \max(r_i(t'), - s_i(t'-1)).
    \end{equation}

    Then the modified policy remains feasible at each time~$t'>t$.  We
    show by induction that, for all $t' \ge t$,

    \begin{equation}
      \label{eq:12}
      \U(t') - \hat{\U}(t') \ge \sum_{i\in S} (s_i(t') - \hat s_i(t')).
    \end{equation}

    For $t' = t$, it is immediate from the
    definitions~\eqref{eq:1},~\eqref{eq:4} and~\eqref{eq:7} (and since
    $\re(t) < 0$), that~\eqref{eq:12} holds with equality.  For
    $t'>t$, assume the result~\eqref{eq:12} is true with $t'$ replaced
    by $t'-1$; we consider two cases:
    \begin{compactitem}[-]
    \item if $\re(t') \ge 0$, then there is no unserved energy at
      time~$t'$ under any feasible policy, so that the left side
      of~\eqref{eq:12} remains unchanged between times $t'-1$ and
      $t'$, while, from~\eqref{eq:11}, the right side of~\eqref{eq:12}
      decreases (weakly) between times $t'-1$ and $t'$; thus the
      inequality~\eqref{eq:12} continues to hold at time~$t'$;
    \item if $\re(t') < 0$, then, between times $t'-1$ and~$t'$, both
      the right and left sides of~\eqref{eq:12} increase by
      $r_i(t') - \hat r_i(t')$, so that~\eqref{eq:12} again continues
      to hold at time~$t'$.
    \end{compactitem}
    It also follows by induction, using~\eqref{eq:11}, that
    $\hat s_i(t') \le s_i(t')$ for all $i\in S$ and $t'\ge t$.  Hence,
    from~\eqref{eq:12}, it again follows that, under the modification
    associated with the time~$t$, the total unserved energy to each
    successive time~$t'\ge t$ does not increase.
  \end{compactenum}
  
  To show the second assertion of the proposition, observe that, under
  the above construction, the greedy policy finally associated with
  each time~$t'$ is defined entirely by the residual energy
  process~$(\re(t),\,t\le t')$ up to and including that time.
\end{proof}

\begin{proof}[Proof of Proposition~\ref{proposition:lp}]

  For each time~$t$, let $\hat r(t) = (\hat r_i(t), \, i\in S)$ be the
  vector of rates determined by the algorithm of the proposition, and
  let $\hat u(t)$ be the corresponding imbalance given
  by~\eqref{eq:4}.  It follows from
  Proposition~\ref{proposition:greedy} that, when the objective is the
  minimisation of total unserved energy over time, it is sufficient to
  consider greedy policies.  Further, for such policies, at those
  times~$t$ such that the residual energy~$\re(t) \ge 0$ the spilled
  energy $u(t)$ is minimised, and at those times~$t$ such that the
  residual energy~$\re(t) < 0$ the unserved energy $-u(t)$ is
  minimised.  It is clear that, at each time~$t$, the
  imbalance~$\hat u(t)$ defined by the above algorithm achieves this
  minimisation in either case.  Thus the problem of choosing, at each
  successive time~$t$, a vector~$r(t)$ of feasible rates so as to
  maximise the expression given by~\eqref{eq:8} reduces to that of
  choosing such a vector~$r(t)$ so as to
  maximise~$\sum_{i\in S}v^t_i(s)r_i(t)$ (where, again the state
  vector $s = s(t-1)$) subject to the additional constraint that the
  corresponding imbalance~$u(t)$ defined by~\eqref{eq:4} is equal to
  $\hat u(t)$.

  Assume, for the moment, that, at the given time~$t$, the ordering of
  states by their charging or discharging priorities is in each case
  unique, i.e. that we do \emph{not} have
  $\eta_i v^t_i(s) = \eta_j v^t_j(s)$ for any $i,j\in S$ or
  $v^t_i(s) = v^t_j(s)$ for any $i,j\in S$.  Then the above vector of
  rates $\hat r(t) = (\hat r_i(t), \, i\in S)$ determined by the given
  algorithm is unique.  Let $r(t) = \bar r(t)$ be the (or any) vector
  of rates which maximises~$\sum_{i\in S}v^t_i(s)r_i(t)$ subject to
  the corresponding imbalance~$\bar u(t)$ being equal to $\hat u(t)$
  as required above.  Let $S_+ = \{i\in S\colon \bar r_i(t) > 0\}$ and
  let $S_- = \{i\in S\colon \bar r_i(t) < 0\}$.  Then the rate
  vector~$\bar r(t)$ satisfies the following four conditions, in each
  case since otherwise the above objective function
  $\sum_{i\in S}v^t_i(s)r_i(t)$ could clearly be increased, while
  maintaining the given imbalance constraint $\bar u(t) = \hat u(t)$:
  \begin{compactenum}[1.]
  \item\label{c:1}%
    subject to the constraint that the total amount charged to the
    stores is as given by $\sum_{i\in S_+}\bar r_i(t)$, both~$S_+$ and
    the individual rates~$\bar r_i(t)$, $i\in S_+$, are as determined
    by the store charging priorities defined by the proposition;
  \item\label{c:2}%
    similarly, subject to the constraint that the total amount
    discharged by the stores is as given by
    $-\sum_{i\in S_-}\bar r_i(t)$, both~$S_-$ and the individual
    rates~$\bar r_i(t)$, $i\in S_-$, are as determined by the store
    discharging priorities defined by the proposition;
  \item\label{c:3}%
    the condition~\eqref{eq:9} is satisfied for all $i\in S_-$,
    $j\in S_+$:
  \item\label{c:4}%
    there are no pairs of stores $i,j \in
    S$ satisfying~\eqref{eq:9} such that it is possible to improve the
    solution $\bar r(t)$ by (further) cross-charging from~$i$ to~$j$.
  \end{compactenum}
  It is now easy to see that the above conditions~\ref{c:1}--\ref{c:4}
  are sufficient to ensure that $\bar r(t)$ is precisely as determined
  by algorithm, i.e.\ that $\bar r(t) = \hat r(t)$.
   
  In the event that, at the given time~$t$, the ordering of states by
  either their charging or discharging priorities is not unique (and
  so $\hat r(t)$ is not unique), it is easy to see that $\bar r(t)$,
  defined as above, may be adjusted if necessary---while continuing to
  maximise $\sum_{i\in S}v^t_i(s)r_i(t)$ subject to the given
  imbalance constraint---so that we again have
  $\bar r(t) = \hat r(t)$: one standard way to do this is to perturb
  $v^t_i(s)$, $i\in S$, infinitesimally so that the given $\hat r(t)$
  becomes the unique solution of the scheduling algorithm, then let
  $\bar r(t)$ solve the given constrained maximisation problem as
  above, and then finally allow the perturbation to tend to zero to
  obtain the required result.
\end{proof}

\end{appendices}

\bibliography{storage_refs_all}

\end{document}